\documentclass[11pt]{amsart}
\usepackage[parfill]{parskip}    
\usepackage{graphicx, txfonts}
\usepackage{epstopdf}
\newtheorem{theorem}{Theorem}[section]

\theoremstyle{definition}
\newtheorem{definition}[theorem]{Definition}
\newtheorem{defn}[theorem]{Definition}

\newtheorem{example}[theorem]{Example}

\newtheorem{problem}[theorem]{Problem}

\theoremstyle{remark}
\newtheorem{remark}[theorem]{Remark}

\usepackage{hyperref}
\usepackage{amsthm, amssymb, amsfonts, amsmath, enumerate, xy, mathrsfs}
\usepackage[margin=1in]{geometry}
\input xy
\xyoption{all}
\newcommand{\C}{\mathcal{C}}

\newcommand{\A}{\mathbb{A}}
\newcommand{\R}{\mathbb{R}}
\newcommand{\Z}{\mathbb{Z}}

\newcommand{\M}{\mathcal{M}}

\DeclareMathOperator{\hocolim}{hocolim}

\renewcommand{\emptyset}{\varnothing}

\renewcommand{\tilde}[1]{\widetilde{#1}}

\newcommand{\cat}[1]{\mathcal{#1}}
\newcommand{\sF}{\mathscr{F}}
\newcommand{\sW}{\mathscr{W}}
\newcommand{\sQ}{\mathscr{Q}}

\DeclareMathOperator{\Hom}{Hom}

\DeclareMathOperator{\dom}{dom}

\DeclareMathOperator{\ho}{Ho}

\newcommand{\po}{\ar@{}[dr]|(.7){\Searrow}}
\newcommand{\pb}{\ar@{}[dr]|(.3){\Nwarrow}}

\title{On Colimits and Model Structures in Various Categories of Manifolds}
\author{David White}
\address{Department of Mathematics \\ Denison University
\\ Granville, OH 43023}
\email{david.white@denison.edu}

\begin{document}
\maketitle

\begin{abstract}
After explaining the importance of model categories in abstract homotopy theory, we provide concrete examples demonstrating that various categories of manifolds do not have all finite colimits, and hence cannot be model categories. We then consider various enlargements of our categories of manifolds, culminating in categories of presheaves. We explain how to produce model structures on these enlarged categories, culminating with answering an open problem involving Poincar\'{e} spaces.
\end{abstract}

\section{Introduction}

Abstract homotopy theory applies the techniques of classical homotopy theory in other areas of mathematics. This technique has led to Fields Medal winning work, and has resolved important open problems in diverse fields such as algebraic geometry, homological algebra, higher category theory, and representation theory, among others. Often, the setting for abstract homotopy theory is that of model categories. A model category has classes of morphisms--the weak equivalences, cofibrations, and fibrations--that satisfy certain axioms (see Definition \ref{defn:model-category}), and allow one to mimic the constructions of classical homotopy theory, such as CW approximation, path spaces, etc. The homotopy theory comes into play by inverting the weak equivalences to make them isomorphisms, which involves passing from the model category to its associated homotopy category.

Proof techniques in (abstract) homotopy theory tend to rely heavily on computation, construction, and induction, e.g., constructing some complicated object (usually via a tower of increasingly complicated objects) and then proving inductively that we can understand what's going on at each step and that in the limit these steps do what is required. For example, this is how the Postnikov tower is used, how Goodwillie calculus (also known as functor calculus) proceeds, and how the small object argument works \cite[Theorem 2.1.14]{hovey-book}. The steps of the construction often involve products, disjoint unions (or, more generally, coproducts), gluing objects together, and taking quotients. Similarly, in equivariant contexts, we often need to quotient by group actions, or study fixed points of group actions. Hence, an important axiom of a model category is that it has all limits and colimits. Even if one does abstract homotopy theory using $\infty$-categories instead of model categories, these fundamental considerations still apply, and hence having (at least finite) limits and colimits is essential to setting up a workable theory.

Since their introduction by Dan Quillen in 1967, researchers have continually expanded the list of examples of model categories. Model categories can be used to study the categories of topological spaces, simplicial sets, chain complexes of modules over a ring (with applications to Andre-Quillen cohomology), the stable module category from representation theory, categories of matrix factorizations, the category of categories itself (as well as groupoids and higher categories), various flavors of (equivariant) spectra (the setting for stable homotopy theory), quasi-coherent sheaves over a scheme, (directed) graphs, (pre)sheaves, dynamical systems (e.g., the theory of flows in computer science), and a plethora of abelian categories that admit cotorsion pairs (leading to a powerful way to study $Ext$ and $Tor$ functors).

Conspicuously absent from this list are the categories of smooth manifolds (required for differential topology), Riemannian manifolds (required for differential geometry), complex analytic spaces, and topological, symplectic, and piecewise-linear manifolds. The purpose of this paper is to explore how one might do abstract homotopy theory in these contexts. This paper grew out of a note that the author wrote in graduate school, and this paper is aimed at graduate students.

It is natural for students to ask whether these various categories of manifolds admit model structures. The biggest obstacle is the lack of colimits. Because these categories lack colimits, it is rare to find a categorical approach to the study of manifolds, and hence it is difficult for students to learn about the behavior of colimits in these categories. For this reason, after providing definitions we need from category theory in Section \ref{sec:background}, in Section \ref{sec:colimits}, we provide examples illustrating the lack of colimits. The differential topology literature is full of work-arounds for the failure of the category of manifold to be closed under certain colimits (such as quotienting by group actions), so in Section \ref{sec:specific-colimits} we explain those work-arounds and why they are insufficient to solve the problem of lack of colimits in general.

Once we realize that any reasonable category of manifold lacks colimits, it is natural to consider enlarged categories of manifolds such as orbifolds, diffeological spaces, and Poincar\'{e} spaces. We consider these options in Section \ref{sec:enlarge}. Finally, in Section \ref{sec:presheaves}, we explain a general procedure for embedding various categories of manifolds into presheaf categories, and endowing them with appropriate model structures. This procedure is inspired by the Fields Medal winning work of Vladimir Voevodsky that launched motivic homotopy theory as a way to use model categories to study the category of schemes in algebraic geometry. This section culminates with a positive answer to a problem posed by John Klein.

\section*{Acknowledgments}

This paper grew out of a draft that I wrote in 2012 while a PhD student at Wesleyan University. I am grateful to Mark Hovey for teaching me about model categories, to Ilesanmi Adeboye's encouragement in writing this paper and for all he taught me about manifolds and Lie groups, to Enxin Wu for teaching me about diffeological spaces and inspiring my interest in them, and to the referee for helpful comments. Lastly, I am grateful to John Klein for sharing the slides of his 2019 talk in Ohio, where he stated Problem \ref{problem:klein}.

\section{Category Theory Background} \label{sec:background}

In this section, we will recall necessary terminology from \cite{Maclane} and \cite{hovey-book}.

Let $\C$ be a category. Let $J$ be a small category (i.e., ob$(J)$ and mor$(J)$ are sets rather than classes). A \textbf{diagram} of type $J$ in $\C$ is a functor $F: J \to \C$. We say $J$ is the \textbf{indexing category} for the diagram. This definition means $F$ picks out an object of $\C$ for each $j\in ob(J)$ and then assigns $\C$-morphisms $F(f)$ for each $J$-morphism $f$.

\begin{example} \label{ex1}
If $J$ is the category $\bullet \gets \bullet \to \bullet$ of three objects and two non-identity morphisms (often called a \textbf{span}), then $F(J)$ consists of three objects in $\cat C$ and two morphisms connecting them, like $C\gets A \to B$.
\end{example}

\begin{definition}
A \textbf{colimit} of a diagram $F:J\to \C$ is an object $L\in \C$, that: 

\begin{enumerate}
\item admits morphisms $\phi_X:F(X)\to L$ for each $X\in J$,
\item for any $f:X\to Y$ in $J$, we have $\phi_Y \circ F(f) = \phi_X$, and
\item if $N$ is any other object in $\C$ that satisfies the two conditions above (i.e., admits morphisms from all elements in the diagram, commuting with the diagram's structure morphisms) then there is a unique morphism $L\to N$ making everything commute.
\end{enumerate}

\end{definition}

The third condition above is known as a \textbf{universal property}. We often say that $L$ is the `closest' object the diagram maps to. This property means that the colimit is unique up to isomorphism. When we say `small colimit' we mean an object that is the colimit of a small indexing category $J$. Note that colimits are sometimes called `direct limits' or `inductive limits' by algebraists. The dual notion of the colimit of a diagram $F:J\to \C$ is the \textbf{limit}, meaning an object that is universal with respect to the property of mapping to $F(J)$. Sometimes the phrase `inverse limit' is used in place of `limit.' In categories of manifolds, limits are almost always computed in the category $Top$ of topological spaces, hence do not pose a problem, so we will focus on colimits. The easiest way for me to remember the difference between colimits and limits is to remember that coproduct is a colimit and product is a limit. The product maps to both factors, just like the limit does. The coproduct is usually the disjoint union of the objects, and all objects map to it by inclusion. A category is called \textbf{cocomplete} if it has colimits for all diagrams coming from small indexing categories $J$. A category is \textbf{bicomplete} if it has all small limits and colimits.

\begin{example}
We return to Example \ref{ex1} of a span $\bullet \gets \bullet \to \bullet$. Below, we display the diagram $F(J)$ with colimit $L = B\coprod_A C$ included, and with the universality demonstrated, looks like:

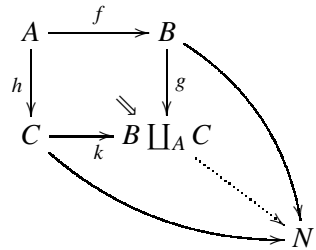
\[
\xymatrix{A \ar[r]^f \ar[d]_h \po & B\ar[d]^g \ar@/^1pc/[ddr] & \\
C \ar@/_1pc/[drr] \ar[r]_k & B\coprod_A C \ar@{..>}[dr]& \\
 & & N}
\]

Here the dotted arrow exists by universality because of the existence of the curved arrows. This diagram is commutative, meaning the two paths from $A$ to $B\coprod_A C$ are equal (i.e., $g\circ f = k\circ h$), the two from $B$ to $N$ are equal, and the two from $C$ to $N$ are equal. The object $B\coprod_A C$ is called the \textbf{pushout} of $f$ and $h$, and the arrow $\Searrow$ is notation to let us know this is a pushout diagram.
\end{example}

Analyzing pushout diagrams is part of the bread-and-butter work of abstract homotopy theory. In the category of topological spaces, if $f$ and $h$ are inclusions, then the object $B\coprod_A C$ is a copy of $B$ glued together with a copy of $C$ along the copy of $A$ that sits in both. Thinking of pushouts and colimits as gluing can be a useful point of view, and helps clarify why homotopy theorists want to work with categories that have all colimits. If $A$ is the initial object, then we simply write $\coprod$ without a subscript, and $B\coprod C$ is the \textbf{coproduct}. In the category of topological spaces, coproduct is disjoint union.

\begin{example}
Another important example of a colimit is a \textbf{coequalizer}, which is obtained from the indexing category with two objects $0$ and $1$ and two non-identity morphisms from $0$ to $1$. A diagram $F(J)$ in $\C$ consists of two objects $A$ and $B$, and two arrows $f,g: A\to B$. The coequalizer is an object $C$ together with $h: B\to C$ such that $h(f(a)) = h(g(a))$, as usual satisfying the universal property that any other $h_2: B\to D$ with this property factors through $C$ (this means, $C$ is the `closest' object to $B$ making $f$ and $g$ equal). One place coequalizers arise in practice is the tensor product. If $R$ is a commutative ring and $M$ and $N$ are $R$-modules, then we can compute the tensor product $M\otimes_R N$ as the coequalizer of the diagram $M \coprod R \coprod N\stackrel{\to}{\to} M\coprod N$ where one of the two maps is the right-action of $R$ on $M$ (i.e., takes $M\coprod R \to M$) and the other is the left-action of $R$ on $N$ (taking $R\coprod N\to N$). The universal property of coequalizer exactly says that one can move multiplication by $r\in R$ across the tensor product, i.e., $m\otimes rn = mr\otimes n$.
\end{example}

\begin{example}
Another important example of an indexing diagram $J$ is the infinite chain $\bullet \to \bullet \to \bullet \to \dots$. For example, a functor $F: J\to Top$ might pick out an infinite collection of spaces and maps $X_n \to X_{n+1}$ illustrating how to glue on one cell at a time. The colimit is the space that the diagram is `trying to be' in the same way that a limit of a sequence of numbers is the number that the sequence is trying to be. We call such a colimit a \textbf{directed colimit}. Note that this indexing diagram is infinite. Some authors only require $\C$ to have \textbf{finite colimits}, meaning colimits of finite diagrams, like the two preceding examples.
\end{example}

A final example of colimit that we'll discuss in more detail describes a group action on an object of $\C$. 

\begin{example}
Let $G$ be a finite group and let the indexing category $J$ be the category $G$ that has one object, $|G|$ arrows (one for each element of the group), and composition law given by the group law (i.e., $f_\sigma \circ f_\tau = f_{\sigma \cdot \tau}$). In this case a functor $F:G \to \C$ simply picks out one object $M\in \C$ and lets $G$ act on $M$ via automorphisms of $M$, one assigned to each group element $g\in G$. Functoriality assures us that the action law $g_1\cdot (g_2 \cdot m) = (g_1*g_2)\cdot m$ is satisfied. In the category of smoooth manifolds, the fact that $G$ sits inside Aut$(M)$ tells us $G$ is isomorphic to a group of diffeomorphisms, so the action map is necessarily smooth even though the group $G$ need not be a Lie group or have a smooth structure. The colimit of this diagram is the object $M/\sim$ where we have identified $m_1$ and $m_2$ if there is some $g$ such that $g\cdot m_1 = m_2$. This object is also known as the \textbf{orbit space} or \textbf{space of coinvariants} $M_G$, and is an important object in equivariant homotopy theory.
\end{example}

We need one more definition. For a cardinal $\kappa$, an object $X$ is called \textbf{$\kappa$-small} if $Hom(X,-)$ commutes with $\kappa$-directed colimits, i.e., colimits of functors from $\kappa$ to $\C$ \cite{hovey-book}.

\begin{definition} \label{defn:locally-pres}
A category is called \textbf{locally presentable} if it has all small colimits and if there is regular cardinal $\kappa$ and a set $S$ of $\kappa$-small objects, such that every object can be obtained as a $\kappa$-filtered colimit of a diagram of objects in $S$.
\end{definition}

For example, the category of sets is locally presentable because every set is a directed colimit of its finite subsets, so we take $\kappa = \aleph_0$ and take $S$ to be the set of finite subsets. Similarly, the category of simplicial sets is locally presentable \cite[Chapter 3]{hovey-book}.

\subsection{Model categories}

Model categories were invented in 1967 by Dan Quillen, whose key observation was that in $Top$ one cares not only about weak equivalences, but also about cofibrations and fibrations \cite{hovey-book}. Cofibrations are morphisms that are used to build more complicated spaces from simpler ones, e.g., morphisms that glue new cells to old spaces. Fibrations are the homotopy theorist's version of projection, e.g., the morphism from the total space to the base space in a vector bundle, or projection from a Lie group onto its quotient by a subgroup. With these classes of morphisms in mind, we now define the notion of a model category.

\begin{defn} \label{defn:model-category}
A {\bf model category} is a category $\M$ with all limits and colimits, and with three distinguished classes of morphisms called weak equivalences $\sW$, cofibrations $\sQ$, and fibrations $\sF$ such that

\begin{itemize}
\item $\sW$ satisfies the 2-out-of-3 property, i.e., if two out of the three morphisms $f, g, g \circ f$ are in $\sW$, then so is the third.

\item $\sW, \sQ, \sF$ are closed under retracts, where $f$ is a retract of $g$ if $f$ and $g$ fit into a commutative diagram where the horizontal composites are $id_A$ and $id_B$ respectively:
\[
\xymatrix{A\ar[r] \ar[d]_f & C\ar[r] \ar[d]^g & A\ar[d]^f \\
B \ar[r] & D \ar[r] & B}
\]
\item Trivial cofibrations (i.e., morphisms in $\sQ \cap \sW$) satisfy the \textit{left lifting property} with respect to fibrations, i.e., when $f \in \sQ \cap \sW$ and $g\in \sF$ then the lift below exists and makes both triangles commute.
\[
\xymatrix{A \ar[r] \ar[d]_f & C\ar[d]^g \\ B\ar[r] \ar@{..>}[ur] & D}
\]
Similarly, cofibrations satisfy the left lifting property with respect to trivial fibrations. Dually, we say fibrations satisfy the \textit{right lifting property}.
\item Any morphism $f$ can be factored into two composites: either a cofibration followed by a trivial fibration or a trivial cofibration followed by a fibration. Furthermore, the assignment of $f$ to any of these factoring morphisms is functorial.
\end{itemize}
\end{defn}

The benefit of model categories is that they have homotopy categories $\ho(\M) := \M[\sW^{-1}]$ and the morphisms in these homotopy categories can be understood via the structure on $\M$. Let $\emptyset$ be the initial object of $\M$ and $\ast$ be the terminal object. If $\emptyset = \ast$ then we say $\M$ is \textit{pointed}.

\begin{defn}
An object $A$ is \textit{cofibrant} if the natural morphism $\emptyset \to A$ is a cofibration. Dually, an object $B$ is \textit{fibrant} if $B\to \ast$ is a fibration. 
\end{defn}

Let $\M_{cf}$ denote the subcategory of cofibrant and fibrant objects of $\M$ (a.k.a., bifibrant objects). There is an equivalence of categories $\M_{cf}/_\sim \cong \ho(\M)$ where $\sim$ is the homotopy relation \cite[Theorem 1.2.10]{hovey-book}.

The classes $\sQ$ and $\sF$ have the added benefit of providing cofibrant and fibrant replacement functors on $\M$. These are methods of functorially replacing objects by nicer objects that are the same up to homotopy (i.e., isomorphic in $\M[\sW^{-1}]$).  

\begin{defn}
For any object $X$ in a model category $\M$, the \textit{cofibrant replacement} of $X$ is the cofibrant object $QX$ and the morphism $QX \to X$ obtained by applying the cofibration-trivial fibration factorization to the natural morphism $\emptyset \to X$. So $QX \to X$ is a trivial fibration with cofibrant domain, and any morphism $f:X\to Y$ yields a morphism $Qf:QX\to QY$ by functoriality. 

Dually, the \textit{fibrant replacement} of an object $A$ is a fibrant object $RA$ together with a trivial cofibration $A\to RA$. Fibrant replacement is obtained by applying the trivial cofibration-fibration factorization to the natural morphism $Y\to \ast$. A morphism $g:A\to B$ gives rise to a morphism $Rg:RA\to RB$ by functoriality.
\end{defn}

In $Top$, cofibrant replacement is cellular approximation. In $Ch(R)$ cofibrant replacement can be projective resolution, or fibrant replacement can be injective resolution. The fact that $Ch(R)$ is a model category allows for hands-on constructions within the derived category $\mathcal{D}(R)$, leading to a better understanding of the morphisms between two objects, the triangulated structure, derived functors, and algebraic invariants such as Andr\'{e}-Quillen cohomology.

The correct notion of a functor between model categories (i.e., one that respects the homotopy-theoretic structure) is that of a left Quillen functor.

\begin{defn}
Let $\M$ and $\cat N$ be model categories and let $F:\M \leftrightarrows \cat N:U$ be an adjoint pair of functors (where $F$ is left adjoint to $U$). We say $F$ is a \textit{left Quillen functor}, if it preserves cofibrations and trivial cofibrations. We say $U$ is a \textit{right Quillen functor} if it preserves fibrations and trivial fibrations.

We say the pair $(F,U)$ is a \textit{Quillen equivalence} if it descends to an adjoint equivalence of categories $\ho(\M)\leftrightarrows \ho(\cat N)$.
\end{defn}

\section{Categories of manifolds lack colimits} \label{sec:colimits}

In this section, we provide concrete examples illustrating that various categories of manifolds lack colimits, and hence cannot be model categories. Because some authors only require model categories to have finite limits and colimits, we focus our counterexamples on finite colimits. We fix the following notation:

\begin{itemize}
\item $Mnfd$ is the category of smooth manifolds, with smooth maps.
\item $TopMnfd$ is the category of topological manifolds, with continuous maps.
\item $LieGp$ is the category of Lie groups.
\end{itemize}

In order to do homotopy theory with manifolds, we'd like to know whether these categories are closed under colimits. The answer is no, as we will now demonstrate.

The first idea for finding a counterexample is to try gluing things together. An easy example is to take two circles and glue them together at a point (this is a pushout in $Top$, of the span $S^1 \gets \ast \to S^1$) to get a space that looks like a figure eight, a space usually denoted $S^1 \vee S^1$. At the point where the two circles are glued together, there is no local neighborhood homeomorphic to $\mathbb{R}$. For the sake of completeness, we show how to check that this space really is the colimit.

\begin{example} \label{ex:pushoutMnfd}
The pushout of $S^1 \gets \{0\} \to S^1$ is $P = S^1 \vee S^1$, which we view as a subspace of $\mathbb{R}^2$, as two circles of radius 1, centered on the points (-1,0) and (1,0), and intersecting at (0,0). To show this, note that the diagram maps to $P$, since we can include the left copy of $S^1$ as the circle around (-1,0) (call this map $i$), the right copy as the circle around (1,0) (call this map $j$), and $\{0\}$ as $\{(0,0)\}$. If the diagram maps to any other object $Y$ (say by maps $f$ from the left copy of $S^1$ and $g$ from the right copy) then there is a unique map from $P$ to $Y$ by sending a point $i(a)$ to $f(a)$ and sending $j(b)$ to $g(b)$, and using the fact that $i(0) = j(0) = (0,0)$ will map to $f(0) = g(0)$. This proves that $P$ has the universal property of the pushout, so $P$ is the pushout. However, $P$ is clearly not a manifold because the origin doesn't have a locally Euclidean neighborhood.
\end{example}

There is one subtlety that the proof above leaves out. We computed the pushout \textit{in the category of topological spaces} and saw that it's not a manifold. But how do we know that the pushout in the category of manifolds is the same as the pushout in the category of topological spaces? Clearly, we need a proof by contradiction, and this can be accomplished in a way analogous to Example \ref{ex:schemes} below. However, because it's easier and more broadly applicable, we now explain a different way to prove that this span does not have a colimit in $Mnfd$. In the proof above, $Y$ was allowed to be any topological space. If we tried to run this proof in $Mnfd$, it wouldn't work because $P$ is not an object in $Mnfd$. If we try to run the proof in $Top$ but where $Y$ is only allowed to be a manifold (and we only consider smooth maps everywhere), then we're secretly relying on an unproven hope that the inclusion of $Mnfd$ into $Top$ respects colimits. That hope is actually false, but this is the right idea. We need to use a category other than $Top$.

One way to prove that a given diagram in $Mnfd$ does not have a colimit, is to embed $Mnfd$, in a colimit-preserving way, into a category that has all colimits, and then prove that the colimit, computed in that category, is not in the image of the embedding.\footnote{The author learned of this strategy from https://mathoverflow.net/questions/19116/} For $Mnfd$ to be a \textbf{full subcategory} of some $\C$ simply means that all the morphisms between two manifolds $M$ and $N$ are still morphisms between their images in $\C$.

For example, the contravariant functor taking a manifold to the space of smooth functions $C^\infty(M,\mathbb{R})$ preserves colimits. Note that when we say a contravariant function preserves colimits, this means it takes colimits to limits. This can be used to prove that the span above does not have a pushout in the category of smooth manifolds (or Lie groups) because if it did, then the resulting manifold would have a point with a 2-dimensional tangent space. That can be verified by letting $I$ be the ideal of functions that vanish at the bad point (the origin in our example above), and then computing that $\dim_{\mathbb{R}} I^2/I^3 = 2$, whereas at every other point, the tangent space has dimension 1. An analogous proof, using the space of continuous functions, works to prove that $TopMnfd$ does not have pushouts in general.

Another category that the category of smooth manifolds embeds into in a colimit-preserving way is the category $\cat F$ of Fr\"{o}licher spaces \cite[Definition 2.5]{stacey}. This $\cat F$ has both limits and colimits, so can be used to address questions about the existence of limits as well. Fr\"{o}licher spaces are determined by smooth maps to/from $\mathbb{R}$, and hence it is easy to argue as in Example \ref{ex:pushoutMnfd} that the pushout of our bad span is $S^1\vee S^1$ in $\cat F$ (because, any smooth function from the pushout to $\mathbb{R}$ factors through the inclusion of the pushout into $\mathbb{R}^2$) and that this object is not in the image of the embedding.

This example of a simple pushout that causes trouble in the local structure of a manifold can be tweaked to provide similar counterexamples showing that the following categories do not have finite colimits and hence cannot have model structures: complex manifolds (lacks the pushout of $\mathbb{C}\gets \{0\}\to \mathbb{C}$), complex analytic spaces, manifolds with boundary, and symplectic, Riemannian, and piecewise-linear manifolds.  A related example is that the pushout of the span $\R \gets (\R - \{0\}) \to \R$, in $Top$, is the line with double origin, and is not even Hausdorff, so cannot be a manifold. However, that span actually does have a pushout in the category of manifolds (namely, $\R$ itself), so we didn't choose this example as our counterexample above.

These example can also be tweaked\footnote{As has been done in https://mathoverflow.net/questions/9961} to show that the category of schemes does not have all colimits. Let $k$ be a field and let $\A^1$ denote the affine line. Recall that the category of affine schemes is isomorphic to the opposite category of $k$-algebras.

\begin{example} \label{ex:schemes}
Consider the diagram Spec$(k(t))\stackrel{\to}{\to} \A^1 \coprod \A^1$, where the two maps are the two inclusions. This diagram does not have a coequalizer in the category of schemes. Morally, the coequalizer is trying to be two copies of the affine line glued together along their generic points or, equivalently, $\A^1$ with all closed points doubled, like the line with two origins discussed above. But this object is too non-separated to be a scheme. This coequalizer can be computed in the category of algebraic spaces, just as we did above enlarging $Mnfd$ to the category of Fr\"{o}licher spaces.

For the sake of contradiction, suppose that this diagram does have a coequalizer, $C$, in the category of schemes, and let $f: \A^1\coprod \A^1 \to C$ be the coequalizer morphism. Let $U$ be an affine open set containing the generic point of $C$. Then $f^{-1}(U)$ yields dense $V_1\coprod V_2$ in $\A^1\coprod \A^1$ and we can define $W = V_1\cap V_2$ inside $\A^1$. When we compute $f(W\coprod W)$ we see that $f$ must take some pair $p,q$ of disjoint closed points to the same point in $U$. However, this violates the universal property, because the map $g$ from $\A^1 \coprod \A^1$ to $\A^1$ with the point $q$ doubled, must factor through $f$, but $g$ separates $p$ and $q$, meaning $f$ cannot send those two points to the same point in $C$.
\end{example}

This example illustrates that the category $Sch_k$ of $k$-schemes cannot be a model category, because it's not cocomplete, and this deficiency inspired Voevodsky's Fields Medal winning approach to algebraic geometry via motivic homotopy theory, which we will discuss in Section \ref{sec:presheaves}. Interestingly, the category $Aff$ of affine schemes over a field $k$ \textit{is} bicomplete, because its opposite category, the category of $k$-algebras, is bicomplete. However, clearly we would not want to do homotopy theory in the category of affine schemes. First, a scheme is glued together from affine schemes using the Zariski topology, and so every non-affine scheme is a colimit of affine schemes in a natural way that the category $Aff$ doesn't `see.' Secondly, while $Aff$ has all colimits, it does not have \textit{homotopy colimits} because when you compute colimits in $Aff$, you can get the wrong answer, geometrically. Specifically, as observed in \cite[Example 2.1.1]{Dugger98}, the span $\A^1 \gets \A^1 - \{0\} \to \A^1$, where the left map is the inclusion $z\mapsto z$ and the right map is $z\mapsto \frac{1}{z}$ has a pushout in both $Aff$ and in $Sch_k$, but they disagree. In $Sch_k$, the pushout is related to the projective line. In $Aff$, the pushout is the terminal object, $Spec(k)$, because $k$ is the intersection of $k[z]$ and $k[z^{-1}]$ inside $k[z,z^{-1}]$. To fix this, we can enlarge $Aff$ by formally adding homotopy colimits, and it turns out we obtain the same enlargement of $Sch$ we get by formally adding colimits, as we will see in Section \ref{sec:presheaves}.

\section{Specific Colimits of Interest in $Mnfd$} \label{sec:specific-colimits}

In this section, we investigate whether putting restrictions on the types diagrams (or on the manifolds therein) can fix the issue identified in the previous section, and result in a class of diagrams that do have colimits. One place this kind of philosophy came up historically was the introduction of homogeneous spaces as manifolds $M$ that have a transitive action of a group $G$. However, what one does with homogeneous manifolds is not to quotient by the group action (yielding $M/G$) but rather to realize $M$ as $G/K$ where $K$ is the stabilizer of a point. This is also a colimit construction, and tells us that homogeneous manifolds are the manifolds that can be realized as a (nice) colimit in the category of Lie groups (since $K$ is a Lie subgroup of $G$ that acts on $G$ by left multiplication). Despite the success of homogeneous manifolds, we will learn that in general, even very nice diagrams of very nice manifolds still do not have colimits.

Thinking about Example \ref{ex:pushoutMnfd}, it is natural to wonder if the issue is that the manifolds appearing have different dimensions (since $S^1$ is a 1-manifold but $\{0\}$ is a 0-manifold). Since the coproduct of two manifolds of dimension $d$ is a manifold of dimension $d$, while the coproduct of manifolds of different dimension is not a manifold, perhaps the colimit of a diagram of manifolds of the same dimension is again a manifold. This turns out to be false, as the following example shows:

\begin{example}
Let $(0,1)$ be the open interval, viewed as a 1-manifold. Let $i$ be the inclusion of $(0,1)$ to $\mathbb{R}$. Note that the pushout of the span $\mathbb{R} \gets (0,1) \to \mathbb{R}$, where both maps are $i$, has a point (namely, $0$) where locally it looks like an intersection of lines, rather than $\mathbb{R}^1$.
\end{example}

This example shows that even very nice diagrams (e.g., with morphisms that are open injections) of connected manifolds of the same dimension can fail to have colimits. 

So far we've shown that $Mnfd$ is poorly behaved with respect to even very nice pushouts. A simple example to show it's poorly behaved with respect to colimits of directed sequences is the sequence $\R \to \R^2 \to \R^3 \to \dots$, where the colimit $\R^\infty$ is not a manifold.

Lastly, we will show that colimits that come from group actions can take us outside the category of manifolds. The subject of group actions on manifolds has been studied much more than the other types of colimits. For example, in Riemannian geometry these considerations give a slick proof that $\pi_1(M)$ is countable for any manifold $M$, by viewing elements of $\pi_1$ as acting via deck transformations. If we plan to mod out by a group action, then for point-set topological reasons (i.e., to get the quotient to be Hausdorff) we need to know that the action is \textbf{properly discontinuous}. This means the action is proper, i.e., $G\times M \to M\times M$ via $(g,m)\mapsto (m,gm)$ is a proper map (inverse images of compact sets are compact) and that $G$ is discrete. We'll assume the action is proper, so that $M/G$ is Hausdorff. Quotients of second countable spaces need not be second countable, but open quotients (i.e., when $\pi:M\to M/G$ is open) are. Because of how $M/G$ is defined, the quotient map is always open. To see this, note that $\pi^{-1}(\pi(U)) = \bigcup_g gU$ is open and $\pi$ is a quotient map, so $\pi(U)$ is open in $M/G$. 

However, even though $M/G$ is Hausdorff and second countable, there are many proper actions where $M/G$ is not a manifold (not even a topological manifold). A concrete counterexample \cite[16.10.3.4]{dieudonne} is the action of $G = \{1,-1\}$ (a group under multiplication) on $M = \mathbb{R}$. The quotient $M/G$ is homeomorphic to the half closed ray $[0,\infty)$ and hence is not a manifold. A positive result is that, for a Lie group $G$, acting on a smooth manifold $M$ by diffeomorphisms (meaning $G\leq Aut(M) = Diffeo(M,M)$), $M/G$ is a differentiable manifold and $\pi: M\to M/G$ is a submersion, if and only if the set of pairs $(x,y)$ where $y=g\cdot x$ for some $g\in G$, is a closed submanifold of $M\times M$ \cite[16.10.3]{dieudonne}. 

An alternative way to get a positive result is to restrict to group actions that are \textbf{free}, meaning fixed-point free (i.e., all stabilizer groups are trivial). The Quotient Manifold Theorem \cite{Lee} says that, if $M$ is a smooth manifold (resp. Lie group) and $G$ acts smoothly, freely, and properly, then $M/G$ is a topological manifold of dimension equal to $\dim M - \dim G$ with a unique smooth structure such that $\pi: M\to M/G$ is a smooth covering map (resp. smooth submersion). Versions of this theorem can also be proven for $C^p$-manifolds rather than $C^\infty$-manifolds.

Indeed, this unique smooth structure makes $M$ into a principal $G$-bundle over $M/G$. Note that here proper but not properly discontinuous is needed. This is one benefit of working in the Lie group context. Another benefit of working in the Lie group context is that any continuous action of a compact Lie group on a topological manifold is proper. However, it is disappointing that even Lie group actions must be free in order to conclude that the orbit space is a manifold. In equivariant homotopy theory, we can rarely reduce our attention to free group actions, and hence the counterexamples in this section illustrate that we will need to enlarge the category of manifolds to a bicomplete category, before we can study manifolds using model categories.

\section{Embedding into a cocomplete category} \label{sec:enlarge}

In this section, we investigate various ways to enlarge the category of manifolds to deal with the lack of colimits. As with the previous sections, we include ideas that do not work, to illustrate the problems with various natural approaches, and warn the reader away from such bad avenues. This section motivates the next section, which explains an idea that works well.

\subsection{Orbifolds}

In the previous section, we saw that it is possible to have a group action $G$ on a manifold $M$ such that the quotient $M/G$ is no longer a manifold. One way to get around this kind of situation is to shift from working with manifolds to orbifolds. Since manifolds are defined locally, perhaps we should allow local group actions, that can be different (and even different groups) in each chart. 

An \textbf{orbifold} is a topological space that is locally a quotient of Euclidean space by the linear action of a finite group. So if one starts with a $G$-manifold $M$ then $M/G$ is an orbifold, but not all orbifolds arise this way because general orbifolds allow different groups to act in different neighborhoods. Furthermore, even though orbifolds fix the problem that $Mnfd$ is not closed under passage to coinvariants, the category $Orbi$ of orbifolds is not cocomplete. So moving from $Mnfd$ to $Orbi$ gives some colimits (those related to group actions) but not all (e.g., pushouts and coequalizers can fail, and our examples from Section \ref{sec:colimits} still apply).

\subsection{Poincar\'{e} Spaces}

John Klein has recently proposed another setting for questions in differential topology. A Poincar\'{e} space (also known as a Poincar\'{e} duality space) is, informally, a space that admits Poincar\'{e} duality. Examples include compact manifolds, homology manifolds (that is, spaces with a homology-isomorphism to a manifold), or the space $(S^2 \vee S^3) \cup_{\alpha} D^5$, which does not have the homotopy type of a manifold. Sometimes, a natural colimit operation on manifolds with boundary will yield a Poincar\'{e} space. For example, the amalgamated union of two $n$-manifolds with boundary is a Poincar\'{e} $n$-complex, as is the quotient of a Poincar\'{e} complex by a free action of a finite group \cite{klein}. We will now give a formal definition.

\begin{definition} \label{defn:poincare}
A topological space $X$ is \textbf{homotopy finite} if it is homotopy equivalent to a finite CW complex. A space \textbf{finitely dominated} if it is a retract of a homotopy finite space. A space is a \textbf{Poincar\'{e} space} of dimension $d$ is a triple $(P,L,[P])$ where $P$ is a finitely dominated space, $L$ is a rank one torsion-free local coefficient system $L$ of abelian groups on $P$, and $[P] \in H_n(P; L)$ is a homology class such that for all local systems $B$, the homomorphism $\bigcap [P]: H^*(P; B) \to H_{d-*}(P; L\otimes B)$, obtained via the cap product, is an isomorphism.
\end{definition}

Note that Poincar\'{e} spaces lack local structure, so many techniques from manifold theory cannot be used in this setting. For a fixed dimension $n$, let $PD$ be the topological category (meaning, enriched in $Top$, which means in this case that the hom-sets are topological spaces) whose objects are Poincar\'{e} spaces of dimension $n$, and where the space of morphisms from $P$ to $Q$ is the space $E(P, Q)$ of Poincar\'{e} embeddings from $P$ to $Q$ \cite{klein}.

Unfortunately, the category $PD$ is not closed under colimits. For example, colimits can build very large spaces out of homotopy finite spaces, that have no chance to be finitely dominated. Nevertheless, Poincar\'{e} spaces may be well-suited to questions in functor calculus, and so we will return to them in Section \ref{sec:presheaves}.

\subsection{Diffeological spaces}

Passing from manifolds to orbifolds, we give up some of the nice behavior we are accustomed to with manifolds, but we obtain more colimits. In this section, we pass from manifolds to diffeological spaces, giving up even more of our geometric intuition, but obtaining a bicomplete category that has a model structure. The category of diffeological spaces can be thought of as sitting in between $Mnfd$ and $Top$, with more geometric information than one has in $Top$.

\begin{defn}(\cite[Definition 2.1]{Wu}) A \textbf{diffeological space} is a set $X$ with a family of maps $U \to X$ (called \textbf{plots}) for each open $U \subset \R^n$ and each n, such that for every open
$U, V \subset \R^n$ we have 
\begin{enumerate}
\item Every constant map $U \to X$ is a plot.
\item If $U \to X$ is a plot and $V \to U$ is smooth, then the composition $V \to X$ is a plot.
\item If $U = \bigcup U_\alpha$ and $U \to X$ is a map such that every restriction $U_\alpha \to X$ is a plot, then $U \to X$ is a plot.
\end{enumerate}
\end{defn}

The morphisms in $Diff$ are maps $X\to Y$ such that for all plots $U\to X$ the composition $U\to Y$ is a plot. The category of smooth manifolds embeds as a full subcategory of $Diff$ where charts are plots and smooth maps of manifolds give maps in $Diff$. Furthermore, $Diff$ is closed under limits and colimits, which is easily seen by noting that the forgetful functor to $Sets$ preserves both limits and colimits because it has both left (discrete diffeology) and right (indiscrete diffeology) adjoints. For example, if $Y$ is a quotient of $X$ then we can declare $f:U\to Y$ to be a plot if for all $p\in U$ there is a neighborhood $V$ where $f$ is of the form $V \to X \to Y$. Similarly, if $A$ is a subset of $X$ then we can declare $f:U\to A$ to be a plot if $U\to A\to X$ is a plot.

In some ways $Diff$ is even nicer than $Top$. For instance, $Diff$ is Cartesian closed, meaning the space of smooth maps between any two diffeological spaces is a diffeological space. This fails for manifolds and fails for topological spaces (unless we restrict attention to compactly generated weak Hausdorff spaces and use the compact-open topology on $C(X,Y)$). Furthermore, $Diff$ is locally presentable but $Top$ is not. Finally, $Diff$ admits lots of geometry, e.g., dimension, differential forms, de Rham cohomology, tangent spaces, tangent bundles, etc. So the category $Diff$ is worth studying in its own right, and also permits the types of constructions (quotients, etc) that homotopy theorists need.

The category $Diff$ has a model structure \cite{kihara} defined in an analogous way to the Quillen model structure on $Top$. However, this model structure is quite difficult to produce, and so far has proven rather difficult to work with, because the standard $p$-simplex $\Delta^p$ does not deformation retract to its horns $\Lambda_k^p$, as occurs in $Top$ \cite{kihara}. This difficulty led Christensen and Wu to consider alternative generating (trivial) cofibrations, and to work out most of the model category axioms for their proposed model structure, but unfortunately, their conjectured model structure fails to satisfy all the model category axioms \cite{Wu}.

It is worth noting that the category of Lie groups sits inside the category of smooth manifolds, and also embeds as a full subcategory of $Diff$. Furthermore, Lie groups are particularly nice in $Diff$; they are fibrant, which means they carry information that is important to the homotopy theory of $Diff$. More generally, any homogeneous space is fibrant and one can prove from this fact that all manifolds in $Diff$ are fibrant, since the way in which $Mnfd$ embeds in $Diff$ sends individual manifolds to homogeneous spaces. The proof of this fact is to note that for any $M$, the group of self-diffeomorphisms $G = Diff(M,M)$ is a smooth group. So letting $H$ be the (smooth) subgroup of diffeomorphisms that fix any point gives $M\cong G/H$. This proof shows the strength of $Diff$, since the $G$ above can be less rigid than a Lie group would be in general, but still works for this construction.

While $Diff$ is a nice solution to the failure of colimits in $Mnfd$ and $LieGp$, it does not help us with topological manifolds, schemes, or Poincar\'{e} spaces. For that reason, we consider in the next section a more general way to enlarge categories of manifolds into bicomplete categories, namely categories of presheaves. Furthermore, this technique can be applied to embed $Diff$ into a category of presheaves that has a model structure with nicer properties than those discussed above \cite{bunk}.

\section{Embeddings into categories of presheaves} \label{sec:presheaves}

In this section, we explain how to embed a category $\C$ into a cocomplete category of presheaves, which can be endowed with a model structure. 

\subsection{Presheaves}

We begin with the basic definition of a presheaf \cite[I.1]{moerdijk}.

\begin{definition} \label{def:presheaf}
Let $\cat D$ be a category. A \textbf{presheaf} on $\cat D$ is a functor $F: \cat D^{op} \to Set$. The category of presheaves $Pre(\cat D)$ has natural transformations as morphisms, and is a locally small category (meaning, there is only a set of morphisms between any pair of objects) if $\cat D$ is a small category.
\end{definition}

The most basic example of a presheaf is a \textbf{representable presheaf} $R_X = Hom(-,X)$, and the Yoneda embedding $\cat D \to Pre(\cat D)$ taking $X$ to $R_X$ is full and faithful, allowing us to view $Pre(\cat D)$ as an extension of $\cat D$. Unfortunately, our categories of manifolds are not small, which inspires the following definition \cite{chorny-white}.

\begin{definition}
If $\cat C$ is a category, a \textbf{small presheaf} on $\cat C$ is a small colimit of representable functors.
\end{definition}

For experts, we mention that small colimits are left Kan extensions of presheaves defined on a small subcategory of $\cat C$ \cite{chorny-white}. We will denote the category of small presheaves by $Pre(\cat C)$, since it is clear from context whether $\cat C$ is a small category or not. The category $Pre(\cat C)$ is the free cocompletion of $\cat C$, meaning it satisfies a universal property identifying it as the `smallest' category that $\cat C$ embeds into that is cocomplete \cite{chorny-white}. Hence, if we want to work in a category with all small colimits, $Pre(\cat C)$ is the most natural candidate. 

If one wants a model category, then the most natural candidate is the category $sPre(\cat C)$ of presheaves into the category of \textit{simplicial} sets. This category can be endowed with the \textbf{projective model structure}, where the weak equivalences (resp. fibrations) of presheaves are natural transformations $N: F\to G$ that are objectwise ($N_X: F(X)\to G(X)$) weak equivalences (resp. fibrations). Cofibrations are then defined by the left lifting property. Alternatively, $sPre(\cat C)$ has the \textbf{injective model structure}, with weak equivalences and cofibrations defined objectwise.

\subsection{Sheaves and sheafification}

We now review sheaves, sheafification, and Grothendieck topologies, following \cite{moerdijk}. Classically, presheaves as in Definition \ref{def:presheaf}, have domain category $\cat D$ the category of open subsets of some topological space $X$, and there are restriction morphisms $r: F(U) \to F(V)$ whenever $V\subset U$; we denote this restriction of $s\in F(U)$ by $s|_V$. A presheaf $F$ is a \textbf{sheaf} when it satisfies locality and gluing axioms with respect to open covers $\{U_i\}_{i\in I}$. Locality means that two sections $s,t\in F(U)$ such that $s|_{U_i} = t|_{U_i}$ for all $i\in I$ must be equal, i.e., it is sufficient to test equality locally. The gluing condition says that any family of sections $s_i \in F(U_i)$ that agree on overlaps, i.e., $s_i|_{U_i\cap U_j} = s_j|_{U_i\cap U_j}$, assembles to a section $s\in F(U)$ such that $s|_{U_i} = s_i$ for all $i\in I$.

Sheafification for more general domain categories $\cat D$ requires a generalization of the notion of an open cover. The analogue of an open cover of an object $X$ is a \textbf{sieve}, defined as a subfunctor of the functor $\Hom(-,X)$ \cite[I.4]{moerdijk}. If $X$ is a topological space, a sieve is a downwards closed subset of open sets. Note that any morphism $f: Y\to X$ in $\cat D$ yields a pullback functor $f^*$ taking sieves on $X$ to sieves on $Y$. 

Next, we require the domain category $\cat D$ of Definition \ref{def:presheaf} to have a \textbf{Grothendieck topology} $\tau$ \cite[Definition III.1]{moerdijk}. That means, for each object $X$ of $\cat D$, we require a collection $\tau(X)$ of sieves, satisfying axioms analogous to being a family of open covers of the objects of $\cat D$. Specifically, we require: 
\begin{enumerate}
\item (covering) for any object $X$, the sieve $\Hom(-,X)$ is in $\tau(X)$,
\item (base change) for any morphism $f: Y\to X$ in $\cat D$ and any $S\in \tau(X)$, the pullback $f^*(S)$ is in $\tau(Y)$, and 
\item (locality) if $S\in \tau(X)$ and if $R$ is a sieve on $X$ such that $f^*(R) \in \tau(Y)$ for all $f: Y\to X$ in $S$, then $R$ is in $\tau(X)$.
\end{enumerate}

The upshot of a Grothendieck topology is that it yields a notion of a \textbf{sheaf} \cite[III.4]{moerdijk} as a presheaf $P: \cat D^{op} \to Set$ such that for every object $X$ of $\cat D$, and each $S\in \tau(X)$, the following diagram is an equalizer  diagram in sets:
\[
\xymatrix{
P(X) \ar[r]^-e & \prod_{f\in S} \limits P(\dom f) \ar@<-.5ex>[r]^-{\strut\smash{p}} \ar@<.5ex>[r]_-{\strut\smash{a}} & \prod_{\substack{f,g \;\; f\in S \\ \operatorname{dom} f = \operatorname{cod} g}} \limits P(\operatorname{dom} g)}
\]
where $f,g$ range over composable pairs, $e(x) = \{P(f)(x)\}_{f\in S}$, $p(\{x_f\}_{f\in S})_{f,g} = x_{f\circ g}$ and $a(\{x_f\}_{f\in S})_{f,g} = P(g)(x_f)$. It is helpful to realize that $p$ arises from the composition in $\cat D$ while $a$ arises from the action relating $\cat D$ and $P$.
The existence of the equalizer $P(C)$ is analogous to the gluing condition, and the fact that it equalizes $p$ and $a$ is analogous to the locality condition. Specialized to the classic case where $\cat D$ is the collection of open subsets of $X$, the two definitions of `sheaf' agree. Sheaves can be described as presheaves that satisfy \textbf{descent} with respect to all covering sieves in $\tau$. The sheaf condition can be enforced by the \textbf{sheafification functor}, defined as the left adjoint to the inclusion functor from the category of sheaves into the category of presheaves. We will see below that sheafification can also be obtained as a Bousfield localization functor, using the framework of model categories.

\subsection{Presheaves of schemes}

In the 1990s, Vladimir Voevodsky wanted to use model categories to learn about schemes, abelian varieties, and motivic cohomology. Topological cohomologies theories, like $K$-theory and $MU$, can be studied using model structures on the category of spectra, the stabilization of the category of spaces. The lack of colimits in the category of schemes, as detailed in Example \ref{ex:schemes}, led Voevodsky to embed the category $Sm/k$ of smooth schemes over $k$ into the category of presheaves of simplicial sets on $Sm/k$, via the Yoneda embedding. Voevodsky endowed this category of presheaves with the injective model structure. The next step is to make the model structure respect sheafification.

For this, one first fixes a Grothendieck topology on the indexing category (here, $Sm/k$), since this determines what is meant by `sheaf.' Next, one defines the collection of local weak equivalences as morphisms of simplicial presheaves that induce isomorphisms on all sheaves of homotopy groups. One then applies left Bousfield localization (discussed in many places, e.g., \cite{batanin-white-eilenberg-moore}) to produce a new model structure whose weak equivalences are the local equivalences, and whose fibrant replacement functor is sheafification. Voevodsky invented this procedure and carried it out with respect to the Nisnevich topology on $Sm/k$ to obtain the model category $Spc$. We can see $X\in Sm/k$ sitting inside $Spc$ as the colimit of any Nisnevich cover of $X$. The Nisnevich topology sits between the Zariski and \'{e}tale topologies, which also yield model structures, but is best suited to the study of motivic cohomology that inspired Voevodsky. Lastly, Voevodsky realized that $\A^1$ should behave like the topological interval and hence should be contractible. The maps $X\times \A^1 \to X$ generate the $\A^1$-weak equivalences of $Spc$, and we again left Bousfield localize to come up with our final model structure for motivic spaces $Spc$, where these morphisms are weak equivalences. Note that the treatment given here follows \cite{Weibel}, whereas originally Voevodsky (\cite{Voev98}) used the category of simplicial sheaves of sets in the Nisnevich topology on $Sm/k$ and defined weak equivalences by way of the functor-of-points approach to schemes. 

With this category $Spc$ in hand, Voevodsky was able to study motivic cohomology and Milnor $K$-theory using model categories. One can stabilize $Spc$ to get a category of motivic spectra, where these cohomologies theories are representable, and Voevodsky was then able to prove the Milnor Conjecture, which relates the Milnor K-theory (mod 2) of a field F of characteristic $\neq 2$ to the Galois cohomology of $F$ with coefficients in $\Z/2\Z$. For this, in 2002, Voevodsky won the Field Medal. 

\subsection{Presheaves of manifolds}

In 1998, inspired by Voevodsky's approach, Dan Dugger sketched a program \cite{Dugger98} to embed the a category $\cat C$ into a category of presheaves, then endow this category with the injective model structure, then localize it with respect to a Grothendieck topology (obtaining the \v{C}ech model structure) so that fibrant replacement is sheafification, and then localize it again so that the interval of $\C$ (for manifolds, this is $\mathbb{R}$) becomes contractible, meaning that all the maps $R_X \times I \to R_X$ is a weak equivalence. The upshot of the \v{C}ech model structure is that, for every hypercover $U\to X$ in the Grothendieck topology, the induced map $\hocolim U_\alpha \to X$ is a weak equivalence. In particular, this implies that every manifold is the colimit of its atlas.

While Dugger's preprint was never completed, most of the ideas it contained have subsequently been worked out. First, in \cite{Dugger01}, Dugger showed how to embed any small category $\cat C$ into the universal model category built from $\cat C$, on the category $sPre(\cat C)$. This is universal in the sense that any map $\gamma: \cat C \to \cat M$ from $\cat C$ to a model category $\cat M$ factors through the Yoneda embedding $\C \to sPre(\cat C)$. When applied to the setting $\cat C = Sm_k$ of the previous section, this produces a model structure Quillen equivalent to Voevodsky's \cite[Proposition 8.1]{Dugger01}. The theory of the \v{C}ech model structure is discussed in \cite[Section 7]{Dugger01} and worked out fully in \cite{dhi}.

Dugger also applies his machinery to topological manifolds. Since the category of topological manifolds is not small, he defines a small indexing category $Man$ consisting of all topological manifolds that are contained in $\R^\infty$, and then proves that his universal model structure on $sPre(Man)$ is Quillen equivalent to the Quillen model structure on topological spaces, so no topological information is lost \cite[Proposition 8.3]{Dugger01}. An alternative to shrinking the indexing category to $Man$ would be to use the theory of small presheaves \cite{chorny-white}.

In Dugger's unfinished manuscript \cite{Dugger98}, he sketches his program for the category of smooth manifolds. The first part of this program was worked out in \cite{pedro}, which produces a model structure on the category of \textit{topological} presheaves, namely functors $F: \C^{op} \to Top$, where $Top$ now denotes the category of compactly generated weak Hausdorff spaces. The authors apply this machinery in the case where $\C$ is the topological category of smooth manifolds of a fixed dimension $d$ and codimension zero embeddings. They endow this $\C$ with a Grothendieck topology given by open covers. They then verify the existence of the \v{C}ech model structure in this context (since Dugger only ever considered \textit{simplicial} presheaves), and confirm that fibrant replacement in this model structure is homotopy sheafification. The authors go on to use this \v{C}ech model structure as a base for the Goodwillie calculus of functors. Again, instead of restricting the size of $\C$, they could have instead used small presheaves \cite{chorny-white}. An analogous program, constructing a model structure for complex manifolds by embedding them into simplicial presheaves then left Bousfield localizing (and relating the resulting model structure to the Oka principle) was carried out in \cite{finnur}. Subsequent work has carried out a similar program for Stein spaces and for complex analytic spaces.

The second part of Dugger's program, inverting the interval and proving that the resulting model category is Quillen equivalent to the category of topological spaces, was carried out by Bunk \cite{bunk}. We note that Bunk replaced the indexing category by the category $Cart$ of Cartesian spaces, i.e., the smooth manifolds that are diffeomorphic to $\R^n$ for some $n$, and considered again simplicial presheaves.

\begin{remark} \label{remark:delta-gen}
One difficulty of working with topological presheaves rather than simplicial presheaves involves local presentability (Definition \ref{defn:locally-pres}). The category of compactly generated weak Hausdorff spaces is not locally presentable, so neither is the category of topological presheaves. For this reason, it is more convenient to work with the \textit{projective} model structure on simplicial presheaves (as is done in \cite{pedro}) and, when doing left Bousfield localization, one uses that this category is \textit{cellular} rather than \textit{combinatorial} (these terms are defined in \cite{white-localization} among other places). This is further discussed in \cite{white-thesis}. An alternative approach, carried out in \cite{gutierrez-white-equivariant} and \cite{hovey-white}, is to work in the category of $\Delta$-generated spaces, which is indeed locally presentable and still contains all CW complexes and all compact Lie groups. The other hypothesis required for left Bousfield localization is left properness (see \cite{bous-loc-semi} for what happens in the absence of this hypothesis) and this is automatically satisfied for either simplicial or topological presheaves (with any of our choices of categories of spaces), as explained in \cite{Reedy-paper, white-commutative-monoids, white-yau}.
\end{remark}

The first part of Dugger's program is largely formal, and can therefore be carried out for other flavors of manifolds, such as those listed in the open problems below, once suitable Grothendieck topologies are identified. However, as noted in \cite{bunk}, the second part of the program (inverting the interval and establishing the Quillen equivalence with spaces) often requires real work. We therefore state the following open problems for readers interesting in gaining practice with the procedure discussed in this section.

\begin{problem}
Carry out the procedure described in this section to produce model structures on simplicial (or topological) presheaves on each of the following categories:

\begin{enumerate}
\item Piecewise-linear manifolds.
\item Symplectic manifolds.
\item Riemannian manifolds.
\item Lie groups.
\end{enumerate}

Subsequently, localize the model categories of presheaves with respect to a Grothendieck topology, then invert the interval, then determine whether or not the resulting model structure is Quillen equivalent to the Quillen model structure on simplicial sets (or topological spaces).
\end{problem}

\subsection{Presheaves of Poincar\'{e} spaces}

We next apply this machinery to the category $PD$ of Poincar\'{e} spaces, answering an open problem posed by John Klein. Let $Top^{PD}$ denote the category of topological functors $F: PD^{op} \to Top$. An example of such a functor is given by the embedding spaces $E(-,N)$ for fixed $N\in PD$. In a 2019 talk at the Ohio State University, Klein stated the following problem:

\begin{problem} \label{problem:klein}
Find a model structure on $\M = Top^{PD}$ in which
\begin{enumerate}
\item the homotopy sheaves are the fibrant objects, and
\item homotopy sheafification is fibrant approximation.
\end{enumerate}
\end{problem}

Here, by `homotopy sheaf,' Klein meant a functor satisfying homotopy descent with respect to a specified covering family that we now define. For fixed $k \leq \infty$, let $D_k$ be the full subcategory of $PD$ whose objects $(U,\partial U)$ have the homotopy type of the disjoint union of $j$ copies of $(D^n,S^{n-1})$ where $0\leq j \leq k$. Define a Grothendieck topology on $PD$ based on the collection of morphisms $U_\alpha \to P$ where $U_\alpha \in D_k$. This Grothendieck topology allows us to define a notion of a homotopy sheaf, as an object satisfying homotopical descent, i.e., $X$ such that $\hocolim U_\alpha \to X$ is a weak equivalence for any covering family $U_\bullet$.

\begin{theorem} \label{thm:poincare-model}
There is a model structure on $Top^{PD}$ satisfying the conditions of Problem \ref{problem:klein}.
\end{theorem}

\begin{proof}
First, note that the site $PD$ is small, thanks to the finiteness conditions in the definition \cite[2.3]{klein}. Hence, the projective model structure on $Top^{PD}$ exists. Every object is fibrant, because every space is fibrant. Every representable presheaf is cofibrant, thanks to the enriched Yoneda Lemma, as discussed in \cite[Section 3]{pedro}. We next consider the collection of morphisms $\hocolim_{S\subset J} \C(-,U_S)\to \C(-,M)$ for each covering $\{U_\alpha \to M\}_{\alpha\in J}$. 
We left Bousfield localize with respect to this collection of morphisms, using that $Top$ is a cellular model category. We note that the proof in \cite[3.2]{pedro} works for any indexing category $\cat C$ and any Grothendieck topology $\tau$. The proof uses that $Top$ is a cartesian model category, but does not require $\cat C$ to be. The fibrant objects of this localization are precisely the homotopy sheaves with respect to the given Grothendieck topology, as in \cite[Theorem 3.6]{pedro}, and fibrant replacement is homotopy sheafification, as in \cite[Remark 3.7]{pedro}.
\end{proof}

In his 2019 talk in Ohio, Klein provided an overview of manifold functor calculus, i.e., the study of contravariant isotopy functors $F: \cat O_P \to Top$ where $\cat O_P$ is the poset of open subsets of a smooth manifold $P$, or to the study of enriched contravariant functors from a topological category of manifolds (with a chosen Grothendieck topology) to $Top$. The goal of manifold calculus is to decompose such functors $F$ into simpler pieces, just like the Taylor series decomposes a function into polynomial pieces. An example of a functor $F$ that one could decompose in this way is the functor $E^{diff}$ that assigns to smooth manifolds $U,N$ the space of smooth embeddings from $U$ to $N$. This functor can be decomposed into a tower of simpler functors (starting with the functor of immersions) that converge to $E^{diff}$, meaning $E^{diff}$ is equivalent to the colimit of the tower. Klein hoped that the model structure of Theorem \ref{thm:poincare-model} would be useful for similarly decomposing the space of Poincar\'{e} embeddings of Poincar\'{e} spaces. We therefore conclude with an open problem, to carry out this program.

\begin{problem} \label{problem:calc}
Use the model structure from Theorem \ref{thm:poincare-model} to set up functor calculus in the Poincar\'{e} setting, and use it to compute the Taylor tower of spaces of Poincar\'{e} embeddings.
\end{problem}

The first part is entirely formal, as has been explained in \cite{chorny-white} among other places. One first further localizes the model structure in Theorem \ref{thm:poincare-model} so that the new fibrant objects are the \textit{homotopy functors}, i.e., functors taking weak equivalences to weak equivalences. For this, it is convenient to replace $Top$ by the model category of $\Delta$-generated spaces as in Remark \ref{remark:delta-gen}, so that results from the theory of locally presentable categories may be used. With the homotopy model structure in hand, the next step is to left Bousfield localize to produce a new model structure whose fibrant objects are the $n$-excisive functors, as in \cite[Theorem 6.1]{chorny-white}, and where fibrant replacement produces the $n^{th}$ layer of the Taylor tower. Together, these model structures produce the Taylor tower. The second part of Problem \ref{problem:calc} will likely require geometric insight and will hopefully yield interesting applications to Poincar\'{e} spaces (such as those discussed in \cite{pedro}), as manifold calculus has done for categories of manifolds.

\end{document}